\documentclass[]{article}

\usepackage[margin=1in]{geometry}
\usepackage{amsmath,amsthm,amssymb,mathrsfs}
\usepackage{graphicx}
\usepackage[font=small]{caption} 

\usepackage[colorlinks=true,linkcolor=blue]{hyperref}

\usepackage[color=orange!30,bordercolor=magenta,colorinlistoftodos]{todonotes}
\presetkeys{todonotes}{size=\small}{}
\newcommand{\C}{\mathbb{C}}

\newcommand{\N}{\mathbb{N}}

\newcommand{\T}{\mathbb{T}}
\newcommand{\W}{\mathbb{W}}
\newcommand{\Z}{\mathbb{Z}}

\newcommand{\qum}[1]{\ensuremath{\widetilde{#1}}}

\newtheorem{Df}{Definition}[section]
\newtheorem{definition}[Df]{Definition}

\newtheorem{proposition}[Df]{Proposition}
\newtheorem{lemma}[Df]{Lemma}

\newtheorem{corollary}[Df]{Corollary}
\newtheorem{conjecture}[Df]{Conjecture}

\newcommand{\overarrowsplus}{\raisebox{-.6ex}{\tikz[{To[]}-,scale=.3]{	\draw (0,1) -- (1,0);\draw[white,-,line width= 3pt] (1,1) -- (0,0);	\draw (1,1) -- (0,0);}}}
\newcommand{\overarrowsminus}{\raisebox{-.6ex}{\tikz[{To[]}-,scale=.3]{	\draw (1,1) -- (0,0);\draw[white,-,line width= 3pt] (0,1) -- (1,0);	\draw (0,1) -- (1,0);}}}
\newcommand{\downcurvearrowleft}{\raisebox{1ex}{\scalebox{-1}{$\curvearrowright$}}}
\newcommand{\downcurvearrowright}{\raisebox{1ex}{\scalebox{-1}{$\curvearrowleft$}}}

\newcommand{\End}{\operatorname{End}}
\newcommand{\Hom}{\operatorname{Hom}}
\newcommand{\Id}{\operatorname{Id}}
\DeclareMathOperator{\Rep}{Rep}
\DeclareMathOperator{\rank}{rank}

\newcommand{\Uhsl}[1]{U_h(\mathfrak{sl}(#1))}
\def\Usl{{{U_h}(\mathfrak{sl}(2|1)})}

\def\e{\operatorname{e}}
\newcommand{\qn}[1]{{\left\{#1\right\}}}
\newcommand{\p}[1]{\ensuremath{\bar {#1}}}
\def\slto{\mathfrak{sl}(2|1)}
\def\slt{\mathfrak{sl}(2)}
\DeclareMathOperator{\SL}{SL}

\newcommand{\cat}{{\Rep_h^{tf}\slto}} 
\newcommand{\unit}{\ensuremath{{\mathrm{1}\mkern-4mu{\mathchoice{}{}{\mskip-0.5mu}{\mskip-1mu}}\mathrm{l}}}}
\newcommand{\coev}{\stackrel{\longrightarrow}{\operatorname{coev}}}
\newcommand{\ev}{\stackrel{\longrightarrow}{\operatorname{ev}}}
\newcommand{\tev}{\stackrel{\longleftarrow}{\operatorname{ev}}}
\newcommand{\tcoev}{\stackrel{\longleftarrow}{\operatorname{coev}}}
\newcommand{\g}{\ensuremath{\mathfrak{g}}}

\newcommand{\A}{\mathbb{D}} 

\title{Quantum invariants arising from $U_h\mathfrak{sl}(2|1)$ are q-holonomic}
\author{Jennifer Brown and Nathan Geer}

\begin{document}

\maketitle

\begin{abstract}
  We show that the quantum invariants arising from typical representations of the quantum group $\Uhsl{2|1}$ are q-holonomic.
  In particular, this implies the existence of an underlying field theory for which this family of invariants are partition functions.
\end{abstract}

\section{Introduction}
q-Holonomic systems were developed to apply techniques from D-modules to the study of hypergeometric functions \cite{Sabbah_1993}.
Such systems appear widely in quantum topology, where they describe the recursive or quasiperiodic behavior of various quantum knot invariants, see \cite{Garoufalidis_Le_2005,Garoufalidis_Le_2016,Garoufalidis_Lauda_Le_2018}.

As in the non-q-deformed setting, q-holonomic systems carry geometric meaning.
In the context of quantum topology, the q-Weyl modules underlying the q-holonomic systems are understood as quantizations of character varieties of three manifolds.

  The precise relationship between the recursion relations and quantum character varieties is the topic of the long-standing AJ conjecture \cite{Garoufalidis_2004, Le_2006}, which posits that the minimal order recurrence relation of the colored Jones polynomial is a quantization of the A-polynomial. 
  Even more geometric in flavor is the volume conjecture and its various generalizations \cite{Kashaev_1997,Gukov_2005,Dimofte_Gukov_2011,Gukov_Sulkowski_Awata_Fuji_2012}, which relate the asymptotic behavior of quantum invariants to the hyperbolic geometry of knot complements. 
  Both conjectures are stated for the colored Jones polynomial but expected to generalize to other families of quantum invariants.
  They both also imply that the invariants are q-holonomic, making this an important property to establish when considering a new family of invariants.

Our choice to focus on $\slto$ is motivated by recent work on Chern-Simons theories constructed from supergroups, such as those studied by Mikhaylov in \cite{Mikhaylov_2015}, see also \cite{Aghaei_Gainutdinov_Pawelkiewicz_Schomerus_2018,Mikhaylov_Witten_2015}.
Perturbative versions of these theories are proposed in \cite{Anghel_Geer_Patureau-Mirand_2021} and lead to quantum invariants of 3-manifolds.
In this paper we study the link invariant underlying the 3-manifold invariants of \cite{Anghel_Geer_Patureau-Mirand_2021}.
 The relationships between these supergroup theories and related topological invariants have been made precise in the case of $ \mathfrak{gl}(1|1)$ in the recent work \cite{Geer_Young_2022}.
 
Unlike the families of representations underlying the colored Jones polynomial and HOMFLY-PT polynomial, typical representations in the superalgebra setting are indexed by a combination of discrete and complex weights and often have vanishing quantum dimensions.
Therefore we must use the techniques developed in \cite{Brown_Dimofte_Garoufalidis_Geer_2020} to prove q-holonomicity in a non-semisimple setting.  
 Having established q-holonomicity, we open the door to a deformed version of a volume conjecture and provide a potential tool towards its study.

  \paragraph{Main result}
This paper is organised around proving the following result:
\smallskip

\textbf{Corollary \ref{cor:slto-is-q-holonomic}} \textit{The quantum invariant built from topologically free typical representations of $\Uhsl{2|1}$ is q-holonomic.}
\smallskip

\noindent At its base, the proof relies on the same closure properties of q-holonomic modules used in \cite{Garoufalidis_Le_2005,Brown_Dimofte_Garoufalidis_Geer_2020}.
Definition \ref{def:q-holonomic-ribbon} and Lemma \ref{lemma:q-hol-from-category} detail sufficient conditions for a collection of objects in a ribbon category to produce q-holonomic Reshetikhin-Turaev invariants.
Section \ref{sec:its-holonomic} is then dedicated to showing that the topologically free typical representations of $\Uhsl{2|1}$ meet these criteria.
This is achieved by writing down explicit formulas for the action of the generators of $\Uhsl{2|1}$, showing them to be composed of q-holonomic functions (Lemma \ref{lemma:q-holonomic-action}), then showing that the R-matrix (Lemma \ref{lemma:R-is-q-holonomic}), evaluation and coevaluation (Lemma \ref{lemma:duality-is-q-holonomic}), and modified dimension (Lemma \ref{lemma:mod-dim-is-q-holonomic}) all preserve this q-holonomicity.

\paragraph{Beyond $\slto$}
We conjecture that q-holonomicity holds for a larger class of Lie superalgebras, see Conjecture \ref{conj:lie-superalgs-holonomic} for a precise statement.
The conjecture is limited to certain well behaved Lie superalgebras for which a proof could follow the same logic and use the same tools as the $\slto$ case.
In this setting the main obstacle is the cumbersome representation theory. 

Further generalisation would likely require a fairly different approach.
This paper largely deals with q-holonomic functions, but the more fundamental object is a q-holonomic module of a q-Weyl algebra.
In the case of quantum invariants, the function is defined on the parameter space of some collection of representations, and the module is defined in relation to these parameters.

To move past this paradigm we need a new object to endow with an action of the q-Weyl algebra.
One candidate is the skein module of the knot complement, which can be q-holonomic in the sense of \cite{Kashiwara_Schapira_2011} as a module for the skein algebra of the torus when the underlying ribbon category is a deformation of a representation category of a reductive group. Recent work \cite{Jordan_Romaidis_2025} has established this is the case for the ribbon category $\Rep_q\SL_2\C$. 

An immediate concern when connecting back to recursion relations is that the skein algebra of the torus is in general not a q-Weyl algebra, or even a quantum torus, and therefore does not closely resemble some algebra of q-difference operators. Section \ref{sec:quantum-torus-from-skein-thy} outlines an approach to modelling the action of q-difference operators using parabolic reduction.

Going further, when the underlying ribbon category is not of the form $\Rep_qG$, then the existing definitions of q-holonomicity do not apply. In this setting a new formulation is needed. One approach could be to take inspiration from Sabbah's definition of homological codimension and formulate q-holonomicity in terms of vanishing $\mathrm{Ext}$ groups.
The author's understanding is that at this question is completely open at the time of publication.

\subsection{A quantum torus from skein theory}\label{sec:quantum-torus-from-skein-thy}
In this section we will use some insights from skein theory, following the ideas of \cite{Gelca_2001,Frohman_Gelca_LoFaro_2001}, to explore the origins of the q-holonomic systems associated with the $\slto$ invariant defined in Section \ref{sec:QuantumInvariants}. 
Some crucial missing ingredients prevent us from drawing a precise line between the action of a skein algebra on a skein module and the various q-Weyl algebra actions considered here.
A full description of the relationship would represent important progress towards a generalised version of the AJ conjecture \cite{Garoufalidis_2004}, and will be left to other works.
In the meantime we explain what is already known or suspected and what is still missing.

We expect that the most topologically-motivated and natural quantum torus and module will be
\begin{equation*}
  \mathrm{SkAlg}_{\Rep_h^{t\!f}\!\mathfrak{t}} (\T) \quad\text{and}\quad \mathrm{SkAlg}_{\Rep_h^{t\!f}\!\slto}(S^3\setminus K),
\end{equation*}
where $\mathfrak{t} \simeq \C[ h_1, h_2 ]$ is the universal Cartan for $\slto$, $K \subset S^3$ is a knot, and $\Rep^{t\!f}_h\!\mathfrak{t}$ is a certain subcategory of the topologically free $U_h\mathfrak{t}$-modules, see Definition \ref{def:topologically-free}.
Note that the quantized enveloping algebra of $\mathfrak{t}$ has an R-matrix given by \eqref{eq:K}, and hence the braiding in the associated category is of the form $v\otimes w \mapsto q^{\omega_{v,w}} w\otimes v$ for some scalars $\omega_{v,w}$.
It's thanks to this that the associated skein algebra will be a quantum torus.
One could expect that the q-Weyl algebra will be a proper sub-torus of this skein algebra, though it is possible that one first needs to deform the braiding.

To describe the action one needs to use parabolic defects for Lie superalgebras, at least implicitly.
Defects for general skein theories are defined in \cite{Brown_Jordan_2025}, following work on quantum decorated character stacks and stratified factorization homology in \cite{Jordan_Le_Schrader_Shapiro_2021,Ayala_Francis_Tanaka_2017}, but the specific construction of parabolic defects for Lie superalgebras has not been undertaken.

The following questions must be resolved to relate a skein theoretic action to the standard q-Weyl algebra action considered here.
\begin{description}
  \item[Question 1:] Given a Borel $\mathfrak{b} \subset \slto$, what's the corresponding surface defect between the $\Rep_qT$ and $\Rep_q SL(2|1)$ theories?
\end{description}
This defect would endow the $\slto$ skein module of the knot complement with an action of the $\mathfrak{t}$ skein algebra of the torus.
\begin{description}
  \item[Question 2:] Does the following procedure recreate the action of the q-Weyl algebra on the $\slto$ quantum invariant?
    \begin{enumerate}
      \item Endow the $\slto$ skein module of a knot complement with an action by the $\mathfrak{t}$ skein algebra of a torus by inserting a parabolic defect parallel to the torus boundary.
      \item Take a relative tensor product with the $\slto$ skein module of the solid torus with a $V(a_1,a_2)$-colored skein along its core and a parabolic defect parallel to its boundary.
      \item Restrict to a subcategory of modules where the two parabolic defects are inverses and then remove them.
      \item Relate the resulting skein in $S^3$ to a scalar multiple of the empty skein, the scalar being (up to normalisation) the quantum invariant of the knot.
    \end{enumerate}
\end{description}
Question 1 is settled for $\mathfrak{sl}(2)$ and unresolved for other Lie algebras and superalgebras.
In the $\mathfrak{sl}(2)$ case, together with the constructions of \cite{Brown_Jordan_2025}, Question 2 presents a fairly complete proof of the AJ conjecture for arbitrary knots.


\subsection{Acknowledgements}

JB is jointly funded by the Simons Foundation award 888988 as part of the Simons Collaboration on Global Categorical Symmetry, and the National Sicence Foundation grant DMS-2202753 and gratefully acknowledges the European Research Council under the European Union's Horizon 2020 research programme under grant agreement No 948885.  NG is partially supported by National Science Foundation grant DMS-2104497.

\section[The Lie superalgebra sl(2|1)]{The Lie superalgebra $\slto$}
We give a brief overview to set notation and recommend \cite{Kac_1977} for a careful treatment of the theory of Lie superalgebras.
Let $\Z_2 = \{\p 0, \p 1\}$ be the additive group of order two.
A \emph{super-space} is a $\Z_{2}$-graded vector space $V=V_{\p 0}\oplus V_{\p 1}$ over $\C$.
The parity of a homogeneous element $x \in V_{\p i}$ is denoted by $\p x=\p i\in \Z_{2}$ and the elements of $V_{\p 0}$ (resp.\ $V_{\p 1}$) are called even (resp.\ odd).
A \emph{morphism of super vector spaces of degree $\p d \in \Z_{2}$} is a linear map $f: V \rightarrow W$ which satisfies $\overline{f(v)} = \p v + \p d$ for each homogeneous $v \in V$.
An algebra on a superspace is a \emph{superalgebra} if its usual structure maps respect the $\Z_2$-grading.
A (left) \emph{module} over a superalgebra $A$ is a super vector space $M$ together with a superalgebra homomorphism $A \rightarrow \End_{\C}(M)$ of degree $\p 0$. 
A \emph{Lie superalgebra} is a super-space $\g=\g_{\p 0} \oplus \g_{\p 1}$ with a super-bracket $[\:
, ] :\g^{\otimes 2} \rightarrow \g$ that preserves the $\Z_{2}$-grading, is super-antisymmetric ($[x,y]=-(-1)^{\p x \p y}[y,x]$), and satisfies the super-Jacobi identity.

Let $A=(a_{ij})$ be the $\slto$ Cartan matrix, i.e. the $2\times 2$ matrix given by $a_{11}=2$, $a_{12}=a_{21}=-1$ and $a_{22}=0$.
\begin{definition}
Let $\slto$ be the Lie superalgebra generated by $h_{i}$, $e_{i},$ and
$f_{i}$, $i=1,2$, where $h_{1}$, $h_{2}$, $e_{1}$ and $f_{1}$ are
even while $e_{2}$ and $f_{2}$ are odd.
The generators satisfy the relations
\begin{align*}
   [h_{i}, h_{j}]&=0,   & [h_{i}, e_{j}] &=a_{ij}e_{j},   &
   [ h_{i},f_{j}] &=-a_{ij} f_{j} &   [e_{i},f_{j}] &=\delta_{ij}h_{i},
\end{align*}
\begin{align*}
 [e_{2},e_{2}] &=[f_{2},f_{2}] =0, & [e_{1},[e_{1},e_{2}]]
 &=[f_{1},[f_{1},f_{2}]]=0.
\end{align*}
\end{definition}
 
Let $h$ be an indeterminate and $q=e^{h/2}$.
We adopt the following standard notations:
\begin{equation*}
  q^z=\e^{zh/2},\quad\qn z=q^z-q^{-z},\quad\text{and}\quad [z] = \frac{\qn z}{\qn 1}
\end{equation*}
In this paper we work with $\Usl$ instead of $U_q(\slto)$ because the standard R-matrix, see \eqref{eq:Rcheck} and \eqref{eq:K}, is well defined as an element in $\Usl \bigotimes \Usl$ but not in $U_q(\slto) \bigotimes U_q(\slto)$.
\begin{definition}\label{D:Usl}
  Let $\Usl$ be the $\C[[h]]$-Hopf
  superalgebra generated by the even elements $h_1,h_2,E_1,$ and $F_1$ together with the odd elements $E_2$ and $F_2$. These are subject to the relations:
\begin{align*}
 [h_{i},h_{j}] &=0, & [h_{i},E_{j}]=&a_{ij}E_{j}, &
 [h_{i},F_{j}]=&-a_{ij}F_{j},
\end{align*}
\begin{align*}
 [E_{i},F_{j}]=&\delta_{i,j}\frac{q^{h_{i}}-q^{-h_{i}}}{q-q^{-1}},  &
 E_{2}^{2}=&F_{2}^{2}=0, 
 \end{align*}
\begin{align*}
  E_{1}^{2}E_{2}-(q+q^{-1})E_{1}E_{2}E_{1}+E_{2}E_{1}^{2}&=0 &
  F_{1}^{2}F_{2}-(q+q^{-1})F_{1}F_{2}F_{1}+F_{2}F_{1}^{2}&=0
\end{align*}
where $[,]$ is the super-commutator given by $[x,y]=xy-(-1)^{\p x \p
 y}yx$. 
The coproduct, counit, and antipode are given by
\begin{align*}
\label{}
   \Delta({E_{i}})= & E_{i}\otimes 1+ q^{-h_{i}} \otimes E_{i}, &
   \epsilon(E_{i})= & 0 & S(E_{i})=&-q^{h_{i}}E_{i}\\
   \Delta({F_{i}}) = & F_{i}\otimes q^{h_{i}}+ 1 \otimes F_{i}, &
   \epsilon(F_{i})= &0 & S(F_{i})=&-F_{i} q^{-h_{i}}\\
   \Delta({h_{i}}) = & h_{i} \otimes 1 + 1\otimes h_{i}, & \epsilon(h_{i})
   = & 0 & S(h_{i})= &-h_{i}.
 \end{align*}
\end{definition}

\subsection[Representations of sl(2|1) and Uhsl(2|1)]{Representations of $\slto$ and $\Uhsl{2|1}$}

In this section we compare the representation theory of $\slto$ and $\Uhsl{2|1}$. 
In particular, we recall the category of \emph{topologically free modules} from which the quantum invariants considered in this paper are constructed.

\begin{definition}\label{def:typical-slto}
  Let $a:=(a_{1},a_{2})\in\Z_{\geq 0}\times \C$. 
We denote by $V(a_1,a_2)$ the irreducible highest weight $\slto$-module with a highest weight vector $v_{0}$ whose parity is even satisfying 
\begin{equation}
  h_{i}.v_{0}=a_{i}v_{0} \qquad \text{and} \qquad e_{i}v_{0}=0
\end{equation}
obtained as the quotient of an infinite dimensional $\slto$ module by a unique maximal sub-module, see \cite[p. 611]{Kac_1978}.
Such a $\slto$ highest weight module is called \emph{\bf typical} if it splits in any finite-dimensional representation, i.e. if $V(a_1,a_2)$ is a sub-module or factor-module of another finite-dimensional $\slto$ module, then it is a direct summand.
\end{definition}
In \cite[Ex 1, p 620]{Kac_1978} this splitting condition is reduced to a system of inequalities on the weights, which for $\slto$ is $a_{1}+a_{2}+1\neq 0$ and $a_{2}\neq 0$.
By requiring $a_1$ to be an integer we've restricted ourselves to just finite dimensional representations \cite[Prop 2.3]{Kac_1978}.
Note that in the same article it's proven that typical modules are projective and injective in the category of finite-dimensional $\slto$-modules. 

\subsubsection[Typical Uhsl(2|1)-modules]{Typical $\Uhsl{2|1}$-modules}\label{sec:TypicalModules}
Next we consider $\Uhsl{2|1}$.
We work over $\C[[h]]$, which is not a field, and so work with the following definition of simplicity, following \cite{Geer_Patureau-Mirand_2018}.

\begin{definition}\label{def:simple}
  A $\Uhsl{2|1}$ module $W$ will be called \emph{\bf simple} if \begin{equation}
    \End_{\Uhsl{2|1}}(W) \simeq \C[[h]].
  \end{equation}
\end{definition}
Note that this \emph{does not imply irreducibility}. For instance $hW \subset W$ is a submodule regardless of the endomorphism ring.
We restrict our attention to the topologically free representations, where the role of the base ring $\C[[h]]$ is well controlled.

\begin{definition}\label{def:topologically-free}
  A $\Usl$-module $W$ is called \emph{\bf topologically free of finite rank} if it is isomorphic as a $\C[[h]]$-module to $V[[h]] := V\otimes \C[[h]]$, where $V$ is a finite-dimensional $\slto$-module.
Let $\cat$ denote the category of topologically free of finite rank $\Usl$-modules where all morphisms are even. 
\end{definition}
As a topological algebra, $\Usl$ is isomorphic to the quantized enveloping superalgebra of $\slto$ \cite{Geer_EKquantization_2006}.
In \cite[Thm 1.2]{Geer_2007} this isomorphism is used to construct a highest weight $\Uhsl{2|1}$-module $\qum{V}(a_1,a_2)$ which deforms $V(a_1,a_2)$.
In particular, it is shown that
\begin{equation}
  \qum{V}(a_1,a_2) \simeq V(a_1,a_2)[[h]] \quad \text{as $\C[[h]$-modules.}
\end{equation}
Note that if $V(a_1,a_2)$ typical (and hence projective and injective) then its h-adic counterpart $\qum{V}(a_1,a_2)$ is likewise both projective and injective.
Similarly, if $V(a_1,a_2)$ is simple, then
\begin{equation}
  \End_{\Uhsl{2|1}}(\qum{V}(a_1,a_2)) \simeq \End_{\slto}(V(a_1,a_2))[[h]] \simeq \C[[h]],
\end{equation}
meaning $\qum{V}(a_1,a_2)$ is likewise simple.
The topologically free representation theory of $\Usl$ is thus parallel to that of the Lie superalgebra $\slto$.

The $\Uhsl{2|1}$-module $\qum{V}(a_1,a_2)$ is called \emph{typical} whenever the associated $\slto$-module $V(a_1,a_2)$ is likewise typical.
Here we give an explicit presentation of typical modules of $\Uhsl{2|1}$.
These are the family of modules used in defining the knot invariant considered in this article.

The typical module $\qum{V}(a_1,a_2)$ has dimension $4a_1 + 4$ and is generated by a distinguished even highest weight vector $v$.
It is uniquely determined by simplicity and the weights $h_1 v = a_1 v$ and $h_2 v = a_2 v$.
We introduce the following notation, motivated by the decomposition \eqref{eq:sl2-decomposition} described below.
\begin{align}
  v_{0,0}  &:= v, &
  v_{-1,1} &:= \left(F_2 F_1 - \frac{[a_1]}{[a_1+1]} F_1 F_2\right)v, &
  v_{1,0}  &:= F_2 v, &
  v_{0,1}  &:= F_2 F_1 F_2 v.
\end{align}
We take as a basis 
\begin{equation}\label{eq:typical-module-basis}
\begin{aligned}
  & F_1^kv_{0,0},   && k = 0,\ldots,  a_1. \\
  & F_1^k v_{-1,1}, && k = 0,\ldots,  a_1 - 1.\\
  & F_1^k v_{1,0},  && k = 0, \ldots, a_1 + 1. \\
  & F_1^k v_{0,1},  && k = 0, \ldots, a_1.
\end{aligned}
\end{equation}
To see that this is indeed a spanning set, we note that the action of $F_1$ and $F_2$ generate $\qum{V}(a_1,a_2)$ and that together $F_2^2 = 0$ and $(q+q^{-1}) F_1 F_2 F_1 = F_1^2 F_2 + F_2 F_1^2$ imply that any product in $\Uhsl{2|1}$ with three or more $F_2$'s is zero.
They also imply that $F_2 F_1 F_2 F_1 = F_1 F_2 F_1 F_2$. 
Furthermore, $\langle h_1, E_1, F_1 \rangle \subset \Uhsl{2|1}$ gives a copy of $\Uhsl{2}$.
As a topologically free finite rank $\Uhsl{2}$-module, $\qum{V}(a_1,a_2)$ decomposes as
\begin{equation}\label{eq:sl2-decomposition}
  \qum{V}(a_1,a_2)[[h]] \simeq V_{a_1}[[h]] \oplus V_{a_1+1}[[h]] \oplus V_{a_1+1}[[h]]\oplus V_{a_1+2}[[h]]
\end{equation}
where $V_n[[h]]$ is the topologically free $\Uhsl{2}$ module that deforms the unique simple $n$-dimensional $\slt$-module $V_n$.
Our choice of basis \eqref{eq:typical-module-basis} is motivated by this decomposition, since each of the four $v_{\epsilon_1,\epsilon_2}$ is a highest weight vector for the induced $\Uhsl{2}$-action.
The subscripts track the $h_1,h_2$ weights, i.e. $h_i v_{\epsilon_1,\epsilon_2} = (a_i + \epsilon_i)v_{\epsilon_1,\epsilon_2}$.

\paragraph{$h_1$ and $h_2$ actions}
Let $k \geq 0$. The $h_1$ action is
    \begin{equation}
      \begin{aligned}
        h_1 &\cdot F_1^kv_{0,0} = (a_1 -2k) F_1^k v_{0,0},\\
        h_1 &\cdot F_1^k v_{-1,1} = (a_1 -2k-1) F_1^k v_{-1,1},\\
        h_1 &\cdot F_1^k v_{1,0} = (a_1 - 2k+1)F_1^k v_{1,0},\\
        h_1 &\cdot F_1^k v_{0,1} = (a_1 - 2k) F_1^k v_{0,1}.
    \end{aligned}
  \end{equation}
  While the $h_2$ action is
  \begin{equation}\label{eq:h2-action}
      \begin{aligned}
        h_2 &\cdot F_1^kv_{0,0} = (a_2 + k) F_1^k v_{0,0}, \\
        h_2 &\cdot F_1^k v_{-1,1} = (a_2 + k+1)F_1^k v_{-1,1}, \\
        h_2 &\cdot F_1^k v_{1,0} = (a_2 + k)F_1^k v_{1,0}, \\ 
        h_2 &\cdot F_1^k v_{0,1} = (a_2 + k + 1) F_1^k v_{0,1}.
    \end{aligned}
  \end{equation}

  \paragraph{$E_1$ and $E_2$ actions}
  Note $E_1 v_{0,0} = E_2 v_{0,0} = 0$ by assumption. Direct computation shows that $E_1 v_{1,0} = E_1 v_{0,1} = 0$, while $E_2 v_{1,0} = [a_2] v_{0,0}$, and $E_2 v_{0,1} = [a_2]F_1v_{0,0} + [a_2-1] F_1 v_{1,0}$. 
  Let $k \geq 1$. The $E_1$ action is
    \begin{equation}
    \begin{aligned}
      E_1 &\cdot F_1^k v_{0,0} = [k][a_1+1-k]F_1^{k-1}v_{0,0},\\
      E_1 &\cdot F_1^k v_{-1,1} = [k][a_1-k] F_1^{k-1} v_{-1,1},\\ 
      E_1 &\cdot F_1^k v_{1,0} = [k][a_1+2-k] F_1^{k-1} v_{1,0},\\
      E_1 &\cdot F_1^k v_{0,1} = [k][a_1+1-k] F_1^{k-1} v_{0,1}.
    \end{aligned}
  \end{equation}
  The $E_2$ action is
    \begin{equation}
    \begin{aligned}
      E_2 &\cdot F_1^k v_{0,0} =  0, \\
      E_2 &\cdot F_1^k v_{-1,1} = \left([a_2+1] - [a_2]\frac{[a_1]}{[a_1+1]}\right) F_1^{k+1} v_{0,0}\\
      E_2 &\cdot F_1^k v_{1,0} = [a_2] F_1^k v_{0,0}, \\
      E_2 &\cdot F_1^k v_{0,1} = [a_2] F_1^{k}F_2F_1 v + [a_2\! +\!1 ] F_1^{k+1} v_{1,0} = [a_2] F_1^k v_{-1,1} + \left(\frac{[a_1][a_2]}{[a_1+1]} + [a_2+1]\right)F_1^{k+1}v_{1,0}. 
    \end{aligned}
  \end{equation}
  For $E_2 \cdot F_1^k v_{0,1}$, we've used that $F_2 F_1 v = v_{-1,1} + \frac{[a_1]}{[a_1+1]} F_1 v_{1,0}$. 

  \paragraph{$F_1$ and $F_2$ actions}
  The $F_1$ action is simply $F_1\cdot F_1^{k}v_{\epsilon_1,\epsilon_2} = F_1^{k+1}v_{\epsilon_1,\epsilon_2}$. 
  It follows from the decomposition \eqref{eq:sl2-decomposition} that
  \begin{equation}
    F_1^{a_1 +1} v_{0,0} = F_1^{a_1}v_{-1,1} = F_1^{a_1+2}v_{1,0} = F_1^{a_1+1} v_{0,1} = 0.
  \end{equation}
 
  The $F_2$ action is more involved.
  By definition $F_2\cdot v_{0,0} = v_{1,0}$, while a computation shows that $F_2 \cdot F_1 v_{0,0} = v_{-1,1} + \frac{[a_1]}{[a_1+1]} F_1 v_{1,0}$.
  We also have $F_2 \cdot v_{-1,1} = -\frac{[a_1]}{[a_1+1]} v_{0,1}$ and $F_2 v_{1,0} =0$.
In what follows $P_n$ is a Laurent polynomial in $q$, determined recursively by $P_0 = 1$, $P_1 = q + q^{-1}$, and $P_n = (q+q^{-1})P_{n-1} - P_{n-2}$.
The $F_2$ action on the rest of the generators is
\begin{equation*}
  \begin{aligned}
    F_2 &\cdot F_1^k v_{0,0} = \left(P_{k-1}F_1^{k-1}F_2F_1 -P_{k-2} F_1^k F_2\right) v = P_{k-1} F_1^{k-1}v_{-1,1} + \left(P_{k-1} \frac{[a_1]}{[a_1+1]} - P_{k-2}\right)F_1^k v_{1,0} &k\geq 2\\
    F_2 &\cdot F_1^k v_{-1,1} = \left(P_{k-1} - \frac{[a_1]}{[a_1+1]} P_k \right)F_1^k v_{0,1} & k \geq 1\\
    F_2 &\cdot F_1^k v_{1,0} = P_{k-1} F_1^{k-1} v_{0,1}, &k\geq 1\\
    F_2 &\cdot F_1^k v_{0,1} = 0. &k \geq 0
  \end{aligned}
\end{equation*}

\subsubsection[The ribbon structure]{The ribbon structure on $\cat$}\label{sec:RibbonStructure}
In Section \ref{sec:QuantumInvariants} we will need that $\cat$ is a \emph{ribbon category}, that is, a braided rigid monoidal category with a natural transformation $\theta : \mathrm{id} \Rightarrow \mathrm{id}$ such that
\begin{equation}\label{eq:twist-constraints}
  \theta_{V\otimes W} = \left(\theta_{V} \otimes \theta_{W}\right) c_{W,V} c_{V,W} \quad\text{and}\quad \left(\theta_V\right)^* = \theta_{V^*}.
\end{equation}
We will describe the ribbon structure on $\cat$.
	
	Since $\Usl$ is a Hopf algebra, $\cat$ is a monoidal category with unit $\unit$ given by the trivial module $\C[[h]]$.  
	Let $\qum{V}=V[[h]]$ be an object of $\cat$ where $V$ is a finite dimensional $\slto$-module.
	The duality is given on objects by $V[[h]]^* = V^*[[h]]$.
  The duality morphisms are $\C[[h]]$-linearly extended from the underlying evaluation and coevaluation maps: 
	 let $\{v_i\}$ be a homogeneous basis of $V$ and $\{v_i^*\}$ be
	the dual basis of $V^*=\Hom_\C(V,\C)$, that is $v_i^*(v_j)=\delta_{i,j}$.
  The following morphisms are defined by $\C[[h]]$-linear extension to give a pivotal structure on $\cat$:
  \begin{equation}\label{eq:pivotal-structure}
  \begin{aligned}
    \coev_{\qum{V}} :\;& \C[[h]] \rightarrow \qum{V}\otimes \qum{V}^{*}, & 1 &\mapsto \sum v_i\otimes v_i^*, & 
    \ev_{\qum{V}}:\; & \qum{V}^*\otimes \qum{V}\rightarrow \C[[h]] , & f\otimes x &\mapsto f(x),\\
    \tcoev_{\qum{V}} :\;& \C[[h]]  \rightarrow \qum{V}^*\otimes \qum{V}, & 1 &\mapsto \sum (-1)^{\p {v}_i} v_i^*\otimes q^{2h_2}v_i, & 
    \tev_{\qum{V}}:\; & \qum{V}\otimes \qum{V}^*\rightarrow \C[[h]], & x\otimes f &\mapsto 
	(-1)^{\p x \p f} f(q^{-2h_2}x).
	\end{aligned}
\end{equation}

Here $x\in V$ and $f\in V^*$ are homogeneous elements.  
	
Khoroshkin-Tolstoy \cite{Khoroshkin_Tolstoy_1991} and Yamane \cite{Yamane_1994} showed that
$\Usl$ has an explicit $R$-matrix that decomposes as $R=\check{R}K$.
Let $E'=E_{1}E_{2}-q^{-1}E_{2}E_{1}$ and $F'=F_{2}F_{1}-qF_{1}F_{2}$.
Then we can write
\begin{equation}
\label{eq:Rcheck}
\check{R}=\exp_{q}(\qn1 E_{1}\otimes F_{1})\exp_{q}(-\qn1 E'\otimes F')\exp_{q}(-\qn1 E_{2}\otimes F_{2}), 
\end{equation}
\begin{equation}
\label{eq:K}
K=q^{-h_{1}\otimes h_{2}-h_{2}\otimes h_{1} -2h_{2}\otimes h_{2}}.
\end{equation}
Here we've used the standard notations $(k)_{q}:=(1-q^{k})/(1-q)$,
$(n)_{q}!:=(1)_{q}(2)_{q}...(n)_{q}$ and $\exp_{q}(x):= \sum_{n=0}^{\infty}x^{n}/(n)_{q}!$.  
By convention $(0)_q! = 1$. 
The identities $\exp_{q}(x)\exp_{q^{-1}}(-x)=1$ and $\exp_{q^{-1}}(x)\exp_{q}(-x)=1$ (see e.g. \cite[(5.2)]{Khoroshkin_Tolstoy_1991} imply that 
\begin{equation}\label{eq:R-inverse}
  R^{-1}=K^{-1}\exp_{q^{-1}}(\qn1 E_{2}\otimes F_{2})\exp_{q^{-1}}(\qn1 E'\otimes F')\exp_{q^{-1}}(-\qn1 E_{1}\otimes F_{1}).
\end{equation}
 Thus, the category $\cat$ has a braiding $\{c_{\qum{V},\qum{W}}\}_{\qum{V},\qum{W}\in \cat}$:
$$
c_{\qum{V},\qum{W}}:\qum{V}\otimes \qum{W}\to \qum{W} \otimes \qum{V}\text{ given by }
v\otimes w\mapsto \tau(R(v\otimes w))
$$
where $\tau$ is the super flip map defined by $\tau(x\otimes y)=(-1)^{\p x \p y}y\otimes x$ for homogeneous $x$ and $y$.  

We finish by showing that the braiding and duality structure are compatible. 
\begin{proposition}\label{P:catBraided}
The above pivotal structure and braiding give a ribbon structure on $\cat$. 
\end{proposition}
	It is well known that $\cat$ is a ribbon category, see for example  \cite{Geer_2005}.
  We give a proof of this fact by describing the ribbon structure in terms of left/right dualities and a braiding.
  Our approach relies on \cite[Theorem 9]{Geer_Patureau-Mirand_2018}, where it is shown that a pivotal braided category is ribbon if it satisfies certain compatibility constraints for a subset of objects on the natural twist morphism defined from the braiding and dualities.
\begin{proof}
In order for \cite[Theorem 9]{Geer_Patureau-Mirand_2018} to apply, $\cat$ must be a generically $\C/\Z$-semisimple pivotal braided category.
Having previously established it's pivotal and braided, it remains to prove the generic semisimplicity condition.

In the language of \cite[\S 1.6]{Geer_Patureau-Mirand_2018}, the category $\cat$ is $\C/\Z$-graded as follows: for $g\in \C/\Z$ let $\cat_g$ be the full subcategory of $\cat$ consisting of those modules whose $h_2$-weights are all equal to $g$ mod $\Z$.
We use the fact that $h_2$ acts semi-simply on each $X$ in $\cat$, decomposing it into $h_2$-weight spaces.
These weights differ by integers, see \eqref{eq:h2-action}.

Next, we show that $\cat$ is generically $\C/\Z$-semisimple, i.e. that the $g$-graded component $\cat_g$ is semisimple whenever $g \in (\C/\Z) \setminus (\Z/\Z)$.
The word ``generic'' refers to the fact that the non-semisimple coset $\Z/\Z$ is symmetric and small, i.e. $-\Z/\Z=\Z/\Z$ and for any $g_1,...,g_n\in \C/\Z$ we have $\bigcup_i (g_i +\Z/\Z) \neq \C/\Z$.

Let $X$ be an object in $\cat_g$, with $g \notin \Z/\Z$.
There exists a highest weight vector $v\in X$.
Since $g \notin \Z/\Z$, the $h_1$ and $h_2$ weights of $v$ satisfy the inequalities of Section \ref{sec:TypicalModules} and so $v$ generates some typical module $\qum{V}(a_1,a_2)\subseteq X$, where $a_2 \equiv g \mod \Z$.
Since typical modules are both projective and injective, $X \simeq \qum{V}(a_1,a_2)\oplus W$ for some $W \in \cat_g$.
Repeating this process on $W$ we see that $X$ is isomorphic to a direct sum of typical modules in $\cat_g$.
We conclude that $\cat_g$ is semisimple whenever $g \notin \Z/\Z$.
Hence $\cat$ is a generically $\C/\Z$-semisimple pivotal braided category.

Next, define a family of natural automorphisms indexed by objects $X$ of $\cat$:
$$\theta_X=(\Id_X \otimes \tev_X)(c_{X,X}\otimes \Id_{X^*})(\Id_X \otimes \coev_X):X\to X.$$
The left-most equality in \eqref{eq:twist-constraints} follows from naturality of the braiding and the Yang-Baxter equation.
By \cite[Theorem 9]{Geer_Patureau-Mirand_2018} the category $\cat$ is ribbon and the natural transformation $\theta$ is called a
\emph{twist} if for all $X \in \cat$: 
\begin{equation}
  \label{eq:twistc}
  \theta_{X^*}= (\theta_X)^* := 
   (\ev_X \otimes  \Id_{X^*})(\Id_{X^*}  \otimes \theta_X \otimes \Id_{X^*})(\Id_{X^*}\otimes \coev_X)
\end{equation}
It is enough to check \eqref{eq:twistc} on generic simple objects, i.e. the typical modules $\qum{V}(a_1,a_2)$ with $a_2 \notin \Z$.
Given an object $X$ in $\cat$, by definition $X^*$ is the $\Usl$-module whose action is given by $y.f(x)=(-1)^{\p y \p f}f(S(y)x)$ for $y\in \Usl$, $x\in X$ and $f\in X^*$.
It follows that the lowest weight vector $v$ of $\qum{V}(a_1,a_2)$ is dual to the highest weight vector $v^*$ in $\qum{V}(a_1,a_2)^*$ and has weight $(a_1, -a_1-a_2-1)$.
Thus $\qum{V}(a_1,a_2)^* \simeq \qum{V}(a_1,-a_1-a_2-1)$.
The morphisms  $ \theta_{\qum{V}(a_1,a_2)}$ and $ \theta_{\qum{V}(a_1,a_2)^*}$ are determined by their values on the highest weight vectors.
A direct computation shows that
\begin{equation}
  \theta_{\qum{V}(a_1,a_2)}=q^{-2a_2(a_1+a_2+1)}\Id_{\qum{V}(a_1,a_2)}
\end{equation}
and
\begin{equation}
  \theta_{\qum{V}(a_1,a_2)^*}=q^{-2(-a_1-a_2-1)(1+ a_1-a_1-a_2-1)}\Id_{\qum{V}(a_1,a_2)^*}.
\end{equation}
The identity $(\ev_X \otimes \Id_{X^*})(\Id_{X^*} \otimes \coev_X)=\Id_{X^*}$ implies that the right side of \eqref{eq:twistc} with $X=\qum{V}(a_1,a_2)$ is equal to $q^{-2a_2(a_1+a_2+1)}\Id_{\qum{V}(a_1,a_2)} $.
It follows that \eqref{eq:twistc} holds for all simple modules in $\cat_g$ with $g\in \C/\Z\setminus \Z/\Z$, thus proving the proposition.
\end{proof}	

\subsection[Quantum Invariants from sl(2|1)]{Quantum Invariants from $\slto$}\label{sec:QuantumInvariants}

\newcommand{\pic}[2]{
	\setlength{\unitlength}{#1}
	{\begin{array}{c} \hspace{-1.3mm}
			\raisebox{-4pt}{#2}
			\hspace{-1.9mm}\end{array}}}
\newcommand{\drawTr}{
	\qbezier(3, 3)(3, 0)(6, 0)
	\qbezier(6, 0)(9, 0)(9, 4)
	\put(9,4){\vector(0,1)3}
	\multiput(0,3)(6,0){2}{\line(0,1){5}}
	\multiput(0,3)(0,5){2}{\line(1,0){6}}
	\qbezier(3, 8)(3, 11)(6, 11)
	\qbezier(6, 11)(9, 11)(9, 7)
}

\newcommand{\drawQDim}{
	\qbezier(0, 3)(0, 0)(3, 0)
	\qbezier(3, 0)(6, 0)(6, 3)
	\put(6,3.5){\vector(0,1)1}
	\qbezier(0, 3)(0, 7)(3, 7)
	\qbezier(3, 7)(6, 7)(6, 4)
}

\newcommand{\Rib}{\mathcal{R}}
\newcommand{\mathsmall}[1]{\mbox{\small$#1$}}
\newcommand{\qdim}{\operatorname{qdim}} 
\newcommand{\md}{\operatorname{\mathsf{d}}}

We established in Section \ref{sec:RibbonStructure} that $\cat$ is a ribbon category.
It is therefore possible to use the frameworks laid out in \cite{Reshetikhin_Turaev_1990,Geer_Patureau-Mirand_Turaev_2009} to define knot invariants with $\cat$ as the algebraic input.
We will not give a full treatment of this theory, and instead defer to the aforementioned articles.
A detailed treatment of the semisimple case can be found in \cite{Turaev_2010}.
	
We consider framed oriented tangles whose components are colored by objects of $\cat$.
Such tangles are called \emph{$\cat$-colored ribbons}.
Let $\Rib_\cat$ denote the category of $\cat$-colored ribbons (for details see \cite{Turaev_2010}).
The well-known Reshetikhin-Turaev construction defines a $\C[[h]]$-linear functor 
  \begin{equation*}
    F:\Rib_\cat\to \cat.
  \end{equation*}
The value of any $\cat$-colored ribbon under $F$ can be computed using the six \emph{building blocks}, which are the morphisms
 $\overarrowsplus,\overarrowsminus,\downcurvearrowleft,\downcurvearrowright,\curvearrowleft,\curvearrowright$ with arbitrary colors in  $\Rib_\cat$.   The functor $F$ transforms these building blocks as follows:
	\begin{equation}\label{eq:building-blocks}
	\begin{array}{l@{\quad}l@{\quad}l}
	F(\overarrowsplus) = c,& F(\downcurvearrowright) = \coev, & F(\downcurvearrowleft) = \tcoev\,, \\[.1cm]
  F(\overarrowsminus) = c^{-1},& F(\curvearrowright) = \ev,& F(\curvearrowleft) = \tev.
	\end{array}
	\end{equation}
  We've suppressed the colors, which determine the components of the braiding and duality transformations.
	Vertical lines are sent to the identity morphism and reversing the direction of an arrow is equivalent to coloring instead by the dual module. 
	
	Let  $L$ be a link with some component labeled by a simple object $V\in \cat$.
  By cutting this component we obtain a tangle $T_V$ whose two ends are labeled with $V$.
  By definition $F(T_V)\in \End_\cat(V)$.
  Since $V$ is simple, this endomorphism is the product of the identity $\Id_V:V\to V$ with an element $\langle T_V \rangle$ of the ground ring of $\cat$, i.e. $F(T_V)= \langle T_V \rangle \Id_V$.
  In particular,
	\begin{equation}\label{eq:DefF}
	\begin{aligned}
	F(L)&=F\left(\pic{0.6ex}{
		\begin{picture}(10,11)(1,0)
		\drawTr
		\put(2,4.5){$\mathsmall{T}$}
		\put(10,4){$\mathsmall{V}$}
		\end{picture}}\;\right)
	=\langle T_V \rangle\,F\left(\pic{0.6ex}{
		\begin{picture}(10,11)(1,0)
		\drawTr
		\put(0.3,4.5){$\mathsmall{\Id_V}$}
		\put(10,4){$\mathsmall{V}$}
		\end{picture}}\;\right)\\
	&=\langle T_V \rangle\, F\left( \pic{0.6ex}{
		\begin{picture}(8,7)(1,0)
		\drawQDim
		\put(7,2){$\mathsmall{V}$}
		\end{picture}}\right) 
	=\langle T_V \rangle (\tev_V\circ \coev_V).
	\end{aligned}
	\end{equation}
  The quantity $\qdim_\cat(V) := \tev_V \circ \coev_V$ is called the \emph{quantum dimension} of $V$.
  The quantum dimension vanishes when $V=\qum{V}(a_1,a_2)$ is a typical representation:
\begin{equation*}
  \qdim_\cat(\qum{V}(a_1,a_2)) :=(\tev_{\qum{V}(a_1,a_2)}\circ \coev_{\qum{V}(a_1,a_2)}) =0\,.
\end{equation*}
	See \cite{Geer_Patureau-Mirand_2007} for details, including a proof of the vanishing quantum dimension. 
  Thus, from \eqref{eq:DefF} we have that $F(L)=0$ if any component of $L$ is colored by a typical module $\qum{V}(a_1,a_2)$.  
	
	In \cite{Geer_Patureau-Mirand_2007} it is shown that one can replace such a vanishing quantum dimension in \eqref{eq:DefF} with a certain non-zero scalar and obtain a well defined invariant of links.
  This process was extended to a general theory in \cite{Geer_Patureau-Mirand_Turaev_2009}.
  We will briefly recall this construction.     
	
	Consider the function $\md$ from the set of typical modules to $\C[[h]]$ given by
	\begin{equation}\label{eq:modified-dimension}
    \md(\qum{V}(a_1,a_2)_{\p p}) = \frac{\qn{a_1+1}}{\qn{1}\qn{a_2}\qn{a_2 +a_1 +1}}
	\end{equation}
	This function is called the \emph{modified dimension}, because it replaces the quantum dimension in the construction of link invariants.	
  It can be derived from $F$ itself applied to open Hopf links, see \cite[\S 2]{Geer_Patureau-Mirand_Turaev_2009}. 
  Let $L$ be a $\cat$-colored framed link with at least one component colored by a typical module $\qum{V}(a_1,a_2)$.
  Cutting this component as above, we obtain a tangle $T_{\qum{V}(a_1,a_2)}$.
By \cite[Prop. 27]{Geer_Patureau-Mirand_Turaev_2009}, the assignment 
\begin{equation}\label{eq:sl21-invariant-1}
  L\mapsto F'(L):=\md(\qum{V}(a_1,a_2)) \langle T_{\qum{V}(a_1,a_2)}  \rangle
\end{equation}
 yields a well-defined isotopy invariant of $L$, independent of the choice of cut component.
 Although \eqref{eq:modified-dimension} may seem ad-hoc, the invariant \eqref{eq:sl21-invariant-1} is unique up to a global scalar.
 See the discussion after \cite[Prop. 27]{Geer_Patureau-Mirand_Turaev_2009}.  
	
 In the remainder of this paper we assume that \emph{every} strand of an $n$-component link $L$ is colored by some typical module $\qum{V}(a_{1i},a_{2i})$, for $i = 1,\ldots, n$.
 It does not matter which strand is cut, and without loss of generality we choose it to be the one labeled by $\qum{V}(a_{11},a_{21})$.
  The corresponding invariant \eqref{eq:sl21-invariant-1} defines a function
\begin{equation}\label{eq:sl21-invariant-2}
  F'(L) : (\Z_{\geq 0} \times \C^*)^n \to \C[[h]]\,,\qquad F'(L)(a_{11},a_{21},...,a_{1n},a_{2n}) = \md(\qum{V}(a_{11},a_{21})) \langle T_{\qum{V}(a_{11},a_{21})} \rangle\,.
\end{equation}
Establishing the $q$-holonomicity of this family of functions is the main result of this paper.

\section{q-Holonomic systems}

We give a brief review of q-holonomic systems and recommend \cite{Sabbah_1993,Garoufalidis_Le_2016} for more details.
Note that we work in the h-adic setting, i.e. over $\C[[h]]$ instead of $\C(q)$, following \cite{Kashiwara_Schapira_2011}.

A prerequisite structure for the formation of a q-holonomic system in this framework is an action of a quantum torus.
The standard story is given by the rank $n$ q-Weyl algebra, which has the presentation
\begin{equation}\label{eq:qWeylAlgebra}
  \mathbb{W}_n := \langle L_1,\ldots,L_n, M_1,\ldots, M_n \; | \; L_i M_j = q^{\delta_{ij}} M_j L_i \rangle
\end{equation}
with the $M$'s and $L$'s commuting between themselves.
This has a standard action on the vector space of functions $\{\Z^n \to \C[[h]] \}$, given by
\begin{align}\label{eq:StandardAction}
  L_if(a_1,\ldots,a_n) = f(a_1,\ldots,a_i + 1 ,\ldots,a_n), && M_i f(a_1,\ldots,a_n) = q^{a_i} f(a_1,\ldots,a_n).
\end{align}
Where $q = e^{h/2}$.

\begin{definition}\cite[Prop. 1.5.2]{Sabbah_1993}
  Let $N$ be a finitely generated module for a q-Weyl algebra $\mathbb{W}_n$.
  The \emph{\bf homological dimension} of $N$ is the degree of its Hilbert polynomial $h_N$ with respect to any good filtration:
  \begin{equation}
    \dim N := \deg h_N.
  \end{equation}
\end{definition}
Note that $\mathbb{W}_n$ itself has homological dimension $2n$, while the trivial module $0$ has homological dimension $0$.
A fundamental result in the theory of D-modules and their q-difference analogs is \emph{Bernstein's Inequality}, which states that either $\rank \mathbb{W}_n \leq \dim M \leq 2\rank \mathbb{W}_n$ or $\dim M = 0$, in which case $M = 0$.
\begin{definition}\label{def:q-holonomic}
  A finitely generated $\mathbb{W}_n$-module $N$ is called \emph{\bf q-holonomic} when its homological dimension is as low as possible, i.e.
  \begin{equation}
    \dim N = n \text{ or } 0.
  \end{equation}
\end{definition}
We will frequently consider cyclic modules $\mathbb{W}_n f(m_1,\ldots,m_n)$.
In this setting lower homological dimension means more linear relations between the functions $P f$, for $P \in \mathbb{W}_n$.
These relations can be interpreted as exhibiting some quasi-periodicity of the function.
When $n=1$ a single non-trivial relation $Pf = 0$ implies q-holonomicity, but
when $n >1$ it is generally not enough to have $n$ distinct operators annihilate $f$.
The following lemma is useful for establishing q-holonomicity of elementary functions:
\begin{lemma}\cite[3.5]{Brown_Dimofte_Garoufalidis_Geer_2020}\label{lemma:q-hol-relations}
  Let $\mathbb{W}_n f$ be a cyclic module whose annihilation ideal contains elements of the form $p_j(M) L_j^{d_j} + q_j(M)$ for $j=1,\ldots,n$, where $p_j, q_j \neq 0$ are independent of the $L_j$ and $d_j \geq 1$.
  Then $\mathbb{W}_n f$ is q-holonomic.
\end{lemma}

\subsection{Non-discrete variables}\label{sec:non-discrete-variables}
We use the same framework as \cite{Brown_Dimofte_Garoufalidis_Geer_2020} to work with complex variables. 
Suppose we have $n$ complex variables $\alpha_1, \ldots, \alpha_n$ and $m$ discrete ($\N$ or $\Z$-valued) ones $a_1,\ldots, a_m$.
Let $\mathbb{V}_n$ be the field of rational functions over $\C[[h]]$, in the variables
$\{x_i \mid i = 1,\ldots, n\}$ and $\{z_{ij} \mid i,j = 1,\ldots, n\}$.
In this way we work formally with the linear and quadratic q-exponentials $q^{\alpha_i}$ and $q^{\alpha_i \alpha_j}$.
Let 
\begin{equation}
  \mathcal{V}_{m,n} := \{ f : \mathbb{Z}^m \to \mathbb{V}_n\}.
\end{equation}
This has an action of $\mathbb{W}_{n+m}$.
For the discrete variables it's given by \eqref{eq:StandardAction}.
For the complex variables, we reproduce this action formally on the co-domain $\mathbb{V}_n$.
Each $M_i$ is multiplication by $x_i$, while the shifts become
\begin{align}
  L_i x_j = q^{\delta_{ij}} x_j &&& L_i z_{jk} = q^{\delta_{ij} \delta_{ik}} x_j^{\delta_{ik}} x_k^{\delta_{ij}} z_{jk}.
\end{align}
In practice it's useful to think of $\mathcal{V}_{m,n}$ as functions $\Z^m \times \C^n \to \C[[h]]$.

\subsection{Examples of q-holonomic functions}
We present some of the functions which show up in the expression of the $\slto$ quantum invariant.
These are all previously known to be q-holonomic.
\begin{description}
  \item[Linear and quadratic powers]
    Functions of the form $q^{n_i^2}, q^{n_i n_j}$, and $q^{n_i}$ are q-holonomic. 
    In \cite{Brown_Dimofte_Garoufalidis_Geer_2020} it's shown that the formal analogs $z_{ii}, z_{ij}$, and $x_i$ are likewise q-holonomic.
  \item[q-Pochhammer]
  The function $n_1,n_2 \mapsto (q^{n_1}; q)_{n_2}$ is q-holonomic by Lemma \ref{lemma:q-hol-relations}, since its annihilation ideal contains $(1-M_{1})L_{1} - 1 + M_{1} M_{2}$ and $L_{2} + M_{1}M_{2} -1$.  
\item[Inverse q-Pochhammer]
  The function $n_1, n_2 \mapsto 1/(q^{n_1};q)_{n_2}$ is q-holonomic by Lemma \ref{lemma:q-hol-relations}.
  Its annihilation ideal contains $(1-M_1M_2)L_1 - M_1 + 1$ and $(1-M_{1} M_{2})L_2 - 1$.
\item[Inverse q-numbers]
  The expression $n \mapsto 1/\{n\}$ is q-holonomic.
  We can see this by rewriting it as a product of q-holonomic functions:
  \begin{equation}
    \frac{1}{\{n\}} := \frac{1}{q^n - q^{-n}} = \frac{q^{-n}}{(q^{-2n};q)_1}
  \end{equation}
  then invoking the closure properties discussed in the next section.
\item[Indicator functions]
  The function 
  \begin{equation}
    \vartheta_{[n_1,n_2]}(n_3) := \begin{cases}
      1 & n_1 \leq n_3 \leq n_2 \\
      0 & \text{otherwise} \\
    \end{cases}
  \end{equation}
  is annihilated by $(M_3-M_1)(L_1-1)$, $(M_3 - q M_2)(L_2-1)$, and $(M_3 - M_1)(M_3 - q^{-1}M_1)(L_3-1)$.
  It follows from Lemma \ref{lemma:q-hol-relations} that it is q-holonomic.
\end{description}

\subsection{Closure properties}\label{sec:closure-properties}
The class of q-holonomic functions enjoy many closure properties, as detailed in \cite[5.2,5.3]{Garoufalidis_Le_2016} and extended to non-discrete variables in \cite[3.2]{Brown_Dimofte_Garoufalidis_Geer_2020}. 
The crucial ones for us here are addition $f+g$, multiplication $fg$, specialization $f|_{k_i = \lambda}$, extension $f(k_1,\ldots,k_{n},k_{n+1}) := f(k_1,\ldots,k_n)$, affine substitution $h(\bar{k}):= f(A \bar{k} + \bar{b})$, and multisums $h(k_1,\ldots,k_{n-1},a,b) := \sum_{k_n=a}^{b} f(k_1,\ldots,k_n)$ .
We will often use these implicitly.

\begin{lemma}\label{lemma:holonomic-matrix-operations}
  Let $X$,$Y$ be parameterized matrices whose coefficients are q-holonomic in both their coordinates (i.e. row and column position) and the parameters. Then the same holds for $X+Y$, $X \oplus Y$, $X \otimes Y$, and $XY$.
\end{lemma}
\begin{proof}
Let $X$ be an $n_1 \times n_2$ matrix and $Y$ an $m_1\times m_2$ matrix.
For matrix sum ($n_1 = m_1, n_2 = m_2$) the coefficients are added.
For direct sum, the new coefficients are 
\begin{equation}
  (X\oplus Y)_i^j = X_i^j\vartheta_{[1,n_1]}(i)\vartheta_{[1,n_2]}(j) + Y_i^j\vartheta_{[n_1+1,n_1+m_1]}(i)\vartheta_{[n_2+1,n_2+m_2]}(j).
\end{equation}
For tensor product the relevant formula is
\begin{equation}
  (X\otimes Y)^{j_1,j_2}_{i_1,i_2} = X_{i_1}^{j_1} Y_{i_2}^{j_2},
\end{equation}
while for matrix product ($n_2 = m_1$) the coefficients are given by
\begin{equation}
  (XY)_i^j = \sum_{k = 1}^{n_2} X_i^k Y_k^j.
\end{equation}
In each case the new coefficients can be written in terms of sums, products, and multisums of the $X_i^j$ and $Y_i^j$, together with indicator functions.
Since these operations preserve q-holonomicity, we conclude that for the four operations considered here, the new coefficients are q-holonomic whenever the original ones are.
\end{proof}

\subsection{q-Holonomic systems and ribbon categories}
Reshetikhin-Turaev style knot invariants, such as those considered in this paper, depend heavily on the choice of a ribbon category. 
It's therefore reasonable to think of their q-holonomicity as a property of this underlying category.
We take a small step in that direction with the following definition and lemma.
\begin{definition}\label{def:q-holonomic-ribbon}
  Fix an $\A$-linear\footnote{If $\A$ is only an integral domain (e.g. $\C[[h]]$) and not a field then we require that hom-spaces are finite rank free $\A$-modules, so that we can still (non-canonically) identify morphisms as matrices. In this case $\dim_\A$ means rank as a free $\A$-module. The main category considered in this paper, $\cat$, meets this requirement.} %
  ribbon category $(\mathcal{C}, c, \tev, \ev, \tcoev,\coev,\theta)$ and a collection $S_\Lambda := \{V_\lambda\}_{\lambda \in \Lambda}$ of its objects.
  Here $\A$ is an integral domain. 
  We require that $S_\Lambda$ is closed under taking duals, i.e. $V_\lambda^* \in S_\Lambda$ for all $\lambda \in \Lambda$.
  Suppose that $\Lambda \hookrightarrow \Z^m\times\C^n$ has codimension zero\footnote{i.e. $\Lambda$ is the complement of a hyperplane arrangement.}, so that $\A$-valued functions on $\Lambda \times \N^k$ can be identified almost-everywhere with elements in $\mathcal{V}_{m+k,n}$.
  We will call the ribbon structure \emph{\bf q-holonomic over $S_\Lambda$} if the following functions exist and are q-holonomic in the sense of \cite{Sabbah_1993,Brown_Dimofte_Garoufalidis_Geer_2020}:

  \begin{description}
    \item[dimension:] $\dim_{\A} : \Lambda \to \mathbb{N}$, given by $\lambda \mapsto \dim_{\A}(V_\lambda)$.
    \item[modified dimension:] $\md : \Lambda \to \A$ given by $\lambda \mapsto \md(V_\lambda)$. Here we assume that the tensor ideal used to define the modified dimension contains $S_\Lambda$. If $\mathcal{C}$ is semisimple, this ideal can be taken to be the full category, in which case modified dimension agrees with the usual quantum dimension of objects.
    \item[braiding:] $\Lambda^2 \times \mathbb{N}^2 \to \A$, given by $(\lambda_1,\lambda_2, i,j) \mapsto \left(c_{V_{\lambda_1},V_{\lambda_2}}\right)_i^j$,
    where the morphism $c_{V_{\lambda_1},V_{\lambda_2}} : V_{\lambda_1}\otimes V_{\lambda_2} \to V_{\lambda_2} \otimes V_{\lambda_1}$ is written in terms of matrix coefficients, with $i = 1,\ldots \dim_{\A}(V_{\lambda_1}\otimes V_{\lambda_2})$ and $j = 1 ,\ldots, \dim_{\A}(V_{\lambda_2}\otimes V_{\lambda_1})$.
      By convention, we extend the row and column indices to $\mathbb{N}$ by zero.
    \item[inverse braiding:]
      As above, based instead on the inverse braiding $\Lambda^2 \times \N^2 \to \A$ given by $(\lambda_1,\lambda_2, i,j) \mapsto \left(c_{V_{\lambda_1},V_{\lambda_2}}^{-1}\right)_i^j$.
    \item[rigidity:] Here we consider four functions $\Lambda \times \mathbb{N} \to \A$:
      \begin{align*}
        (\lambda,i) & \mapsto \left(\coev_{V_\lambda}\right)^{i}, & (\lambda,i) & \mapsto \left(\ev_{V_\lambda}\right)_{i}, \\
        (\lambda, i) & \mapsto \left(\tcoev_{V_\lambda}\right)^{i}, & (\lambda,i) & \mapsto \left(\tev_{V_\lambda}\right)_{i}.
      \end{align*}
      In all cases $i=1,\ldots, \dim_{\A}(V\otimes V^*)$ is extended to $\N$ by zero.
  \end{description}
\end{definition}

In Definition \ref{def:q-holonomic-ribbon}, we're using the restriction to $\Gamma$ of the standard action of $\W_{m+n}$ on functions from $\Z^m\times\C^n$ to $\mathcal{V}_{m+k,n}$.

\begin{lemma}\label{lemma:q-hol-from-category}
  Let $\mathcal{C}$ be a ribbon category.
  If $\mathcal{C}$'s ribbon structure is q-holonomic over a collection $S_\Lambda := \{V_\lambda\}_{\lambda \in \Lambda}$ of simple objects then the corresponding quantum invariants will be q-holonomic in the sense of \cite{Sabbah_1993,Brown_Dimofte_Garoufalidis_Geer_2020}.
\end{lemma}
Lemma \ref{lemma:q-hol-from-category} implicitly relies on the fact that the actions in Definition \ref{def:q-holonomic-ribbon} agree with those on the resulting quantum invariant.
This follows from the fact that the quantum invariant is obtained from the various functions in Definition \ref{def:q-holonomic-ribbon} by a series of summations, products, and convolutions. At each step the new module's action is induced from the old, ensuring agreement.
\begin{proof}
  Let $\mathcal{D}_\lambda$ be a diagram of a link $L$, whose components are colored by simple objects $V_{\lambda_1},\ldots,V_{\lambda_\ell} \in S_\Lambda$, and let $\mathcal{T}_\lambda$ be the same diagram, cut along some arc and turned into a (1,1)-tangle.
  Here we use the shorthand $\lambda = (\lambda_1, \ldots, \lambda_\ell)$. 
  Without loss of generality, we suppose the open ends of the tangle are colored by $V_{\lambda_1}$. 
  As a morphism in $\Rib_\mathcal{C}$, the tangle $\mathcal{T}_\lambda$ can be decomposed into the composition of tensor products of the six basic building blocks: $\overarrowsplus,\overarrowsminus,\curvearrowleft, \curvearrowright, \downcurvearrowleft, \downcurvearrowright$ and identity morphisms on the various $V_{\lambda_i}$. 

  By assumption the evaluation functor $\Rib_\mathcal{C} \to \mathcal{C}$ used to define the quantum invariant sends each of these building blocks to a q-holonomic function.
  Since each $V_{\lambda_i}$ is finite dimensional and the hom-spaces are free $\A$-modules, each of these can be understood as a matrix.
  Since Lemma \ref{lemma:holonomic-matrix-operations} shows that q-holonomicity is preserved by tensor product and composition (which becomes matrix multiplication), it follows that the function $\lambda_1,\ldots,\lambda_{\ell} \mapsto \langle T_\lambda \rangle \in \End_\mathcal{C}(V_{\lambda_1}) \simeq \A$ is q-holonomic.
  The quantum knot invariant is the product of this scalar and the modified dimension, and therefore q-holonomic.
\end{proof}

We wrap up this section with a conjecture regarding the expected generalization of our main result.
\begin{conjecture}\label{conj:lie-superalgs-holonomic}
  Let $\mathfrak{g}$ be a contragredient Lie superalgebra, meaning it's simple, finite dimensional, basic\footnote{Basic here means it's not one of the strange Lie superalgebras $\mathfrak{p}(n)$ or $\mathfrak{q}(n)$, for which the classification of typical representations \cite[Theorem 1]{Kac_1978} doesn't apply.}, and its even sub-algebra is reductive.
  Let $U_h\mathfrak{g}$ be its quantized universal enveloping algebra, and $\Rep_h^{tf}\mathfrak{g}$ the category of topologically free finite rank $U_h\mathfrak{g}$ modules. (See Definition \ref{def:topologically-free}.) 
  Let $S_\Lambda$ denote the set of representations whose underlying $\mathfrak{g}$ representation is typical \cite[\S 4]{Kac_1978}.

  We conjecture that for any knot $K$ the quantum knot invariant $L_K : \Lambda \to \C[[h]]$ built from $S_\Lambda$ using the ribbon structure \cite{Khoroshkin_Tolstoy_1991} and the modified dimension associated to any tensor ideal containing $S_\Lambda$ are q-holonomic.
\end{conjecture}

The proof of this conjecture would follow the same lines as our work here, with significantly more involved expressions for the action of the generators on the typical representations.

\section[The sl(2|1) invariant is q-holonomic]{The $\slto$ invariant is q-holonomic}\label{sec:its-holonomic} 

We show that $\Rep_h^{tf}\mathfrak{sl}(2|1)$ meets the conditions in Definition \ref{def:q-holonomic-ribbon}.
With that established, Lemma \ref{lemma:q-hol-from-category} immediately applies q-holonomicity of the $\slto$ quantum invariant.

We do not specify the ranks of the various q-Weyl algebras appearing in this section. One can take them to be the number of independent variables in the relevant functions.

\begin{lemma}\label{lemma:q-holonomic-action}
  The matrix coefficients of elements in $\langle E_1, E_2, F_1, F_2\rangle$ inside $\Uhsl{2|1}$ in each $\End(\qum{V}(a_1,a_2))$ form q-holonomic systems.
\end{lemma}
Let $\rho_{a_1,a_2} : \Uhsl{2|1} \to \End_{\C[[h]]}(\qum{V}(a_1,a_2))$ be a typical representation.
The claim is that the functions $\N^3 \to \C[[h]](x_1,z_{11})[a_2,a_2^{-1}]$ defined by $(a_1,k,\ell) \mapsto \rho_{a_1,a_2}(X)_k^\ell$ are q-holonomic for any $X \in \langle E_1, E_2, F_1, F_2\rangle$.
Here $x_1$ and $z_{11}$ play the part of the formal expressions $q^{a_2}$ and $q^{a_2^2}$, as outlined in Section \ref{sec:non-discrete-variables},  $k,\ell = 0, \ldots, 4 a_1 + 3$ run over the basis elements \eqref{eq:typical-module-basis}, and $\rho_{a_1,a_2}(X) w_k = \sum_{\ell=0}^{4a_1 + 3} \rho_{a_1,a_2}(X)_k^\ell w_\ell$. 
The following proof relies heavily on the many closure properties of q-holonomic functions, see Section \ref{sec:closure-properties}.

\begin{proof}
  We will make use of the decomposition \eqref{eq:sl2-decomposition} of $\qum{V}(a_1,a_2)$ into four $\Uhsl{2}$-representations.
  In each of the four blocks, our basis is of the form $F_1^k v_{\epsilon_1,\epsilon_2}$ for $\epsilon_1,\epsilon_2 \in \{-1,0,1\}$ and $k = 0, \ldots a_1 + \epsilon_1$.
  For generators which respect the decomposition, q-holonomicity within the blocks implies it for the whole representation, since we build the full expressions from the $\slt$ ones by shifting their indices, multiplying them by the appropriate indicator functions, and adding the results.

  \paragraph{$\boxed{F_1}$}
  Within each $\Uhsl{2}$-subrepresentation, $\rho(F_1)_k^\ell$ has the form 
  \begin{equation}
    (a_1,a_2,k,\ell) \mapsto \delta_{k+1}^{\ell}\vartheta_{[0,a_1+\epsilon]}(k)
  \end{equation}
  where $\vartheta$ is the indicator function.
  Since indicator functions and their products are q-holonomic, we conclude that $F_1$'s coefficients in each typical representation are q-holonomic functions.

  \paragraph{$\boxed{E_1}$}
  We again work in the decomposition \eqref{eq:sl2-decomposition}.
  In each block, $\rho(E_1)_k^\ell$ has the form
  \begin{equation}
    (a_1,a_2,k,\ell) \mapsto [a_1+\epsilon +1-k] \delta_{k-1}^\ell \vartheta_{[0,a_1+\varepsilon]}(k)
  \end{equation}
  By closure under affine substitution, the expression $[n]= \{n\}/\{1\}$ sends linear functions to q-holonomic ones.
  We conclude that the coefficients of $E_1$ are q-holonomic in each typical representation.

  \paragraph{$\boxed{F_2}$}
  The action of this generator does not respect the decomposition \eqref{eq:sl2-decomposition}.
  To describe its coefficients we put a total ordering on the generators, first the $F_1^kv_{0,0}$, then the $F_1^kv_{-1,1}$, then the $F_1^kv_{1,0}$, and finally the $F_1^kv_{0,1}$.
  We will temporarily refer to the $\ell^{\text{th}}$ element in this order as $w_\ell$, so that $w_0 = v_{0,0}$, $w_{a_1+1} = v_{-1,1}$, $w_{2a_1 + 1} = v_{1,0}$, and $w_{3a_1 + 3} = v_{0,1}$. 

  The coefficients of $F_2$ on a typical representation $\qum{V}(a_1,a_2)$ are defined by the piecewise function
  \begin{equation}
    (a_1,a_2,k,\ell) \mapsto \begin{cases}
      P_{k-1} \delta^\ell_{a_1+k} + \left(P_{k-1}\frac{[a_1]}{[a_1+1]} - P_{k-2}\right) \delta^\ell_{2a_1+k+1}
      & 0 \leq k \leq a_1 \\
      \left(P_{k-a_1-2} - \frac{[a_1]}{[a_1+1]} P_{k-a_1-1}\right) \delta^\ell_{2 a_1 + 2 + k}
      & a_1 + 1 \leq k \leq 2 a_1 \\
      P_{k-2a_1-2}^{}\;\delta^\ell_{a_1+1+k}
      & 2a_1 + 1 \leq k \leq 3a_1 +2 \\
      0
      & 3 a_1 + 3 \leq k \leq 4 a_1 + 3 \\
      0 & \text{otherwise.}
    \end{cases}
  \end{equation}

  Since a piecewise q-holonomic function is likewise q-holonomic, we conclude that the coefficients of $F_2$ are q-holonomic.

  \paragraph{$\boxed{E_2}$}
  This generator does not respect the decomposition \eqref{eq:sl2-decomposition}.
We reuse the basis $\{w_\ell\}$ defined in the $F_2$ section of this proof.
The coefficients of $E_2$ on a typical representation $\qum{V}(a_1,a_2)$ are
  \begin{equation}
    (a_1,a_2,k,\ell) \mapsto \begin{cases}
      0
      & 0 \leq k \leq a_1 \\
      \left([a_2+1] - \frac{[a_1][a_2]}{[a_1+1]}\right) \delta^\ell_{k-a_1}
      & a_1 +1 \leq k \leq 2a_1 \\
      [a_2] \delta^\ell_{k-2 a_1}
      & 2a_1 + 1 \leq k \leq 3a_1 + 2 \\
      [a_2] \delta^\ell_{k-2a_1-2} + \left(\frac{[a_1][a_2]}{[a_1+1]} + [a_2+1]\right) \delta^\ell_{k-a_1-2}
      & 3a_1 + 3 \leq k \leq 4a_1 + 3 \\
      0 & \text{otherwise.}
    \end{cases}
  \end{equation}
  We conclude that the actions of the generators are in terms of q-holonomic functions.
  Finally, by Lemma \ref{lemma:holonomic-matrix-operations}, every element in $\Uhsl{2|1}$ acts by a matrix whose coefficients form a q-holonomic system.
\end{proof}

\begin{lemma}\label{lemma:R-is-q-holonomic}
  The $R$-Matrix given by the product of \eqref{eq:Rcheck} and \eqref{eq:K} has q-holonomic coefficients, as does its inverse \eqref{eq:R-inverse}.
\end{lemma}
Our claim is that the function $\N^6 \to \mathbb{V}_2$, sending 
\begin{equation}
  (a_1,k_a,\ell_a,b_1,k_b,\ell_b) \mapsto (R)_{k_a,k_b}^{\ell_a,\ell_b} \in \C[[h]](q^{a_2}, q^{b_2}, q^{a_2^2}, q^{b_2^2}, q^{a_2b_2})
\end{equation}
defined by the action of $R$ on $\qum{V}(a_1,a_2)\otimes \qum{V}(b_1,b_2)$ is q-holonomic.
Since the functions coming from the $R$-matrix and its inverse differ only by negative signs in exponential expressions, it's sufficient to show the result just for $R$.
\begin{proof}
  We consider the four factors $\exp_q(-\{1\}E_2\otimes F_2)$, $\exp_q(-\{1\}E' \otimes F')$, $\exp_q(\{1\}E_1\otimes F_1)$, and $K$.
  First, note that 
      \begin{equation}
        \exp_q(-\{1\}E_2\otimes F_2) = 1 -\{1\} E_2 \otimes F_2\quad\text{and}\quad %
        \exp_q(-\{1\}E'\otimes F') = 1 + -\{1\} E'\otimes F'.
      \end{equation}
      It follows from Lemma \ref{lemma:q-holonomic-action} and Lemma \ref{lemma:holonomic-matrix-operations} that these act by q-holonomic functions. 
    For the $E_1\otimes F_1$ factor, the relevant functions come from the expressions
    \begin{equation}
      E_1^n \cdot F_1^k v_{\epsilon_1,\epsilon_2} = \begin{cases}
        0 & 1 \leq k \leq n\\
        \frac{[k]!}{[k-n]!}\frac{[a_1 + \epsilon_1 - (k+1) +n]!}{[a_1+\epsilon_1 - (k+1)]!} F_1^{k-n} v_{\epsilon_1,\epsilon_2} & \text{otherwise}
      \end{cases}
      \end{equation}
      Note that $1/[m]! = q^{m(m+1)/2} \{1\}^m (q^2 ; q^2)_m$ is q-holonomic for $m$ any discrete variable or linear combination thereof
      and $F_1^n \cdot F_1^k v_{\epsilon_1,\epsilon_2} = F_1^{n+k} v_{\epsilon_1,\epsilon_2}$.
      We conclude that the coefficients
      \begin{equation}
        (a_1,k_a,\ell_a,b_1,k_b,\ell_b,n) \mapsto (E_1^n \otimes F_1^n)_{k_a,k_b}^{\ell_a,\ell_b} \in \mathbb{V}_2
      \end{equation}
      defined by the action on $\qum{V}(a_1,a_2)\otimes \qum{V}(b_1,b_2)$ are q-holonomic.
      Since q-holonomicity is preserved by multisums and $1/(n)_q!$ is q-holonomic, we conclude that $\exp_q(\{1\} E_1\otimes F_1)$ is q-holonomic.
  
      Finally, $K = q^{-h_1\otimes h_2 - h_2\otimes h_1 - 2h_2 \otimes h_2}$ acts on $\qum{V}(a_1,a_2)\otimes \qum{V}(b_1,b_2)$ by 
      \begin{equation}
        (a_1,k_a,\ell_a,b_1,k_b,\ell_b) \mapsto q^{-a_1b_2 - a_2b_1 -2a_2b_2} \delta_{k_a}^{\ell_a} \delta_{k_b}^{\ell_b} \in \mathbb{V}_2.
      \end{equation}
      These are a combination of quadratic exponentials (in mixed discrete and continuous variables) together with indicator functions.
      It follows that the coefficients of $K$ in any typical representation induce q-holonomic systems.

      We've shown that each of the four factors of the R-Matrix is q-holonomic.
      It follows from Lemma \ref{lemma:holonomic-matrix-operations} that the R-Matrix has q-holonomic coefficients in each typical representation.
\end{proof}

\begin{lemma}\label{lemma:duality-is-q-holonomic}
  The evaluation and coevalution induce q-holonomic systems.
\end{lemma}
We are claiming that the following four functions $\N^3 \to \mathbb{V}_1$ are q-holonomic:
\begin{align*}
  (a_1,\ell_1,\ell_2) &\mapsto \left(\coev_{\qum{V}(a_1,a_2)}\right)^{\ell_1,\ell_2}
                      & (a_1,\ell_1,\ell_2) &\mapsto \left(\ev_{\qum{V}(a_1,a_2)}\right)_{k_1,k_2} \\
  (a_1,\ell_1,\ell_2) &\mapsto \left(\tcoev_{\qum{V}(a_1,a_2)}\right)^{\ell_1,\ell_2}
                      & (a_1,k_1,k_2) &\mapsto \left(\tev_{\qum{V}(a_1,a_2)}\right)_{k_1,k_2}
\end{align*}

\begin{proof}
  The functions in question are 
  \begin{align*}
    (a_1,\ell_1,\ell_2) &\mapsto \delta^{\ell_1,\ell_2}
                            & (a_1,k_1,k_2) &\mapsto \delta_{k_1,k_2} \\
    (a_1,\ell_1,\ell_2) &\mapsto (-1)^{\overline{w}_{\ell_1}} q^{2 \phi(\ell_1)} \delta^{\ell_1,\ell_2} 
                            & (a_1,k_1,k_2) &\mapsto (-1)^{\overline{w}_{k_1}\overline{w}_{k_2}} q^{-2 \phi(k_1)} \delta_{k_1,k_2}
  \end{align*}
  where $\phi(\ell)$ is the $h_2$-weight of the basis vector $w_\ell$ defined in the proof of Lemma \ref{lemma:q-holonomic-action}.
  It's a piecewise linear function in $a_1,a_2$, and $\ell$.
  Each expression is a product of standard q-holonomic functions and therefore q-holonomic.
\end{proof}

\begin{lemma}\label{lemma:mod-dim-is-q-holonomic}
The modified dimension for typical modules in $\cat$ is q-holonomic
\end{lemma}
\begin{proof}
  The expression \eqref{eq:modified-dimension} is a product of terms $\{n\}$ and $1/\{n\}$, modified by affine substitution in continuous and discrete variables.
  It is therefore q-holonomic.
\end{proof}

\begin{corollary}\label{cor:slto-is-q-holonomic}
  The $\slto$ quantum invariant is q-holonomic.
\end{corollary}
\begin{proof}
  This follows from Lemma \ref{lemma:q-hol-from-category}, which is applicable following Lemmas \ref{lemma:R-is-q-holonomic}, \ref{lemma:duality-is-q-holonomic}, and \ref{lemma:mod-dim-is-q-holonomic}. 
\end{proof}

\bibliographystyle{alpha}
\bibliography{references.bib}

\end{document}